\documentclass[11pt]{article}
\usepackage[english,francais]{babel}
\usepackage[utf8]{inputenc} 
\usepackage{latexsym,amsmath,amssymb,textcomp}
\usepackage[all]{xy}
\usepackage[nottoc, notlof, notlot]{tocbibind}
\usepackage{geometry} 
\usepackage{graphicx} 
\usepackage{mathtools}
\usepackage{amsthm}
\usepackage{amsmath}
\usepackage{booktabs}
\usepackage{array} 
\usepackage{paralist}
\usepackage{verbatim} 
\usepackage{subfig} 
\usepackage{ stmaryrd }
\usepackage{hyperref}
\usepackage{fancyhdr} 
\title{Quasi-hereditary property of double Burnside algebras.\\
Propriété quasi-héréditaire des algèbres de Burnside doubles.}
\author{Baptiste Rognerud}
\begin{document}
\maketitle
\theoremstyle{plain}
\newtheorem{theo}{Theorem}[section]
\newtheorem{theox}{Theorem}[section]
\renewcommand{\thetheox}{\Alph{theox}}
\newtheorem{prop}[theo]{Proposition}
\newtheorem{lemma}[theo]{Lemma}
\newtheorem{coro}[theo]{Corollary}
\newtheorem{question}[theo]{Question}
\newtheorem*{notations}{Notations}
\theoremstyle{definition}
\newtheorem{de}[theo]{Definition}
\newtheorem{dex}[theox]{Definition}
\theoremstyle{remark}
\newtheorem{ex}[theo]{Example}
\newtheorem{re}[theo]{Remark}
\newtheorem{res}[theo]{Remarks}
\renewcommand{\labelitemi}{$\bullet$}
\newcommand{\comu}{co\mu}
\newcommand{\mo}{\ensuremath{\hbox{mod}}}
\newcommand{\Mo}{\ensuremath{\hbox{Mod}}}
\newcommand{\decale}[1]{\raisebox{-2ex}{$#1$}}
\newcommand{\decaleb}[2]{\raisebox{#1}{$#2$}}
\newcommand\Ind{{\rm Ind}}
\newcommand\Inf{{\rm Inf}}
\newcommand\Res{{\rm Res}}
\newcommand\Def{{\rm Def}}
\newcommand\Iso{{\rm Iso}}
\newcommand\field {{k}}
\newcommand\Mod {{\rm Mod}}
\newcommand\Hom {{\rm Hom}}
\newcommand\mb{\hbox{-}}
\selectlanguage{english}
\begin{abstract}
In this short note we investigate some consequences of the \emph{vanishing} of simple biset functors. As corollary, if there is no non-trivial vanishing of simple biset functors (e.g. if the group $G$ is commutative), then we show that $kB(G,G)$ is a \emph{quasi-hereditary} algebra in characteristic zero. In general, this not true without the non-vanishing condition, as over a field of characteristic zero, the double Burnside algebra of the alternating group of degree $5$ has infinite global dimension. \\
\selectlanguage{francais}
% Text of abstract in French
\begin{center}{\bf  R\'esum\'e}\end{center}
Dans cette note on s'intéresse à quelques conséquences du phénomène dit de \emph{disparition} des foncteurs à bi-ensembles simples. On démontre que dans le cas où il n'y a pas de disparitions non triviales de foncteurs simples (par exemple si le groupe est commutatif) alors l'algèbre de Burnside double en caractéristique zéro est quasi-héréditaire. Sans l'hypothèse de non-disparitions triviales, ce résultat est en général faux. En effet, l'algèbre de Burnside double du groupe alterné de degré $5$ en caractéristique zéro est de dimension globale infinie.
\end{abstract}
\selectlanguage{english}
\par\noindent
{\it {\footnotesize A.M.S. subject classification: 19A22, 20C99, 16G10,18E10.}}
\begin{notations}
Let $k$ be a field. We denote by $\mathcal{C}_{k}$ the biset category. That is the category whose objects are finite groups and morphisms are given by the double Burnside module (see Definition $3.1.1$ of \cite{bouc_biset}).
For a finite group $G$, we denote by $\Sigma(G)$ the full subcategory of $\mathcal{C}_{k}$ consisting of the subquotients of $G$. If $\mathcal{D}$ is a $k$-linear subcategory of $\mathcal{C}_{k}$, we denote by $\mathcal{F}_{\mathcal{D},k}$ the category of $k$-linear functors from $\mathcal{D}$ to $k\mb\Mod$. If $L$ is a subquotient of $K$, we write $L\sqsubseteq K$ and if it is a proper subquotient, we write $L\sqsubset K$. If $V$ and $W$ are objects in the same abelian category, we denote by $[V:W]$ the number of subquotients of $V$ isomorphic to $W$. 
\end{notations}
\section{Evaluation of functors.}
Let us first recall some basic facts about the category of biset functors. Let $\mathcal{D}$ be an admissible subcategory of $\mathcal{C}_{k}$ in the sense of Definition $4.1.3$ of \cite{bouc_biset}. The category $\mathcal{D}$ is a skeletally small $k$-linear category, so the category of biset functors is an abelian category. The representable functors, also called Yoneda functors, are projective, so this category has \emph{enough projective}. The simple functors are in bijection with the isomorphism classes of pairs $(H,V)$ where $H$ is an object of $\mathcal{D}$ and $V$ is a $k\mathrm{Out}(H)$-simple module (see Theorem $4.3.10$ of \cite{bouc_biset}). 
\newline A biset functor is called \emph{finitely generated} if it is a quotient of a \emph{finite} direct sum of representable functors. In particular, the simple biset functors and the representable functors are finitely generated. As in the case of modules over a ring, the choice axiom has for consequence the existence of maximal subfunctor for finitely generated biset functors. If $F$ is a biset functor, the intersection of all its maximal subfunctors is called the radical of $F$ and denoted $\mathrm{Rad}(F)$. 
\newline If $G$ is an object of $\mathcal{D}$, then there is an evaluation functor $ev_{G} : \mathcal{F}_{\mathcal{D},k} \to \mathrm{End}_{\mathcal{D}}(G)\mb\Mod$ sending a functor to its value at $G$. It is obviously an exact functor and it is well known that it sends a simple functor to $0$ or a simple $\mathrm{End}_{\mathcal{D}}(G)$-module. It turns out that the fact that a simple functor vanishes at $G$ has some consequences for the functors having this simple as quotient. 
\begin{prop}\label{rad}
Let $F\in\mathcal{F}_{\mathcal{D},k}$ be a finitely generated functor and let $G\in Ob(\mathcal{D})$. Then
\begin{enumerate}
\item $\mathrm{Rad}(F(G))\subseteq [\mathrm{Rad}(F)](G)$
\item If none of the simple quotients of $F$ vanishes at $G$, then $\mathrm{Rad}(F(G)) = [\mathrm{Rad}(F)](G)$. 
\end{enumerate}
\end{prop}
\begin{proof}
Let $M$ be a maximal subfunctor of $F$. Then $M(G)$ is a maximal submodule of $F(G)$ if the simple quotient $F/M$ does not vanish at $G$ and $M(G)=F(G)$ otherwise. For the second part, if $N$ is a maximal submodule of $F(G)$, let $\overline{N}$ be the subfunctor of $F$ generated by $N$. There is a maximal subfunctor $M$ of $F$ such that $\overline{N}\subseteq M\subset F$. We have $\overline{N}(G)=N\subseteq M(G)\subset F(G)$. By maximality, $M(G)=N$. The result follows. 
\end{proof}
\begin{re}
In Section $9$ of $\cite{bst1}$, the Authors gave some conditions for the fact that the evaluation of the radical of the so-called standard functor is the radical of the evaluation. The elementary result of Proposition \ref{rad} gives new lights on this section. Indeed, Proposition $9.1$ \cite{bst1} gives a sufficient condition for the non vanishing of the simple quotients of these standard functors. \end{re}
Over a field, the category of finitely generated projective biset functors is Krull-Schmidt in the sense of \cite{krause_ks} (Section $4$), so every finitely generated biset functor has a projective cover. 
\begin{coro}\label{eval}
Let $F\in\mathcal{F}_{\mathcal{D},k}$ be a finitely generated functor and let $G\in Ob(\mathcal{D})$. Then,
\begin{enumerate}
\item If $F$ has a unique quotient $S$, and $S(G)\neq 0$, then $F(G)$ is an indecomposable $\mathrm{End}_{\mathcal{D}}(G)$-module.
\item If $P$ is an indecomposable projective biset functor such that $\mathrm{Top}(P)(G)\neq 0$, then $P(G)$ is an indecomposable projective $\mathrm{End}_{\mathcal{D}}(G)$-module. 
\end{enumerate}
\end{coro}
\section{Highest weight structure of the biset functors category.}
Let us recall the famous theorem of Webb about the highest weight structure of the category of biset functors.
\begin{theo}[Theorem $7.2$ \cite{webb_strat2}]
Let $\mathcal{D}$ be an admissible subcategory of the biset category. Let $k$ be a field such that $\mathrm{char}(k)$ does not divide $|\mathrm{Out}(H)|$ for $H\in Ob(\mathcal{D})$. If $\mathcal{D}$ has a finite number of isomorphism classes of objects, then $\mathcal{F}_{\mathcal{D},k}$ is a highest weight category. 
\end{theo}
The set indexing the simple functors is the set, denoted by $\Lambda$, of isomorphism classes of pairs $(H,V)$ where $H\in Ob(\mathcal{D})$ and $V$ is an $k\mathrm{Out}(H)$-simple module. Let $H$ and $K$ be two objects of $\mathcal{D}$. Then $$\bigoplus_{\underset{X\sqsubset H}{X\in \mathcal{D}}} \Hom_{\mathcal{D}}(X,K)\Hom_{\mathcal{D}}(H,X),$$ can be viewed as a submodule of $\Hom_{\mathcal{D}}(H,K)$ via composition of morphisms. We denote by $I_{\mathcal{D}}(H,K)$ this submodule and by $\overleftarrow{\Hom}_{\mathcal{D}}(H,K)$ the quotient $\Hom_{\mathcal{D}}(H,K)/I_\mathcal{D}(H,K)$. This is a natural right $k\mathrm{Out}(H)$-module. If $V$ is an $k\mathrm{Out}(H)$-module, we denote by $\Delta_{H,V}^{\mathcal{D}}$ the functor
\begin{equation*}
\Delta_{H,V}^{\mathcal{D}} := K\mapsto \overleftarrow{\Hom}_{\mathcal{D}}(H,K)\otimes_{k\mathrm{Out}(H)}V. 
\end{equation*}
When the context is clear, we simply denote by $\Delta_{H,V}$ this functor. If $(H,V)\in \Lambda$, then $\Delta_{H,V}$ is a standard object of $\mathcal{F}_{\mathcal{D},k}$. The set $\Lambda$ is ordered by $(H,V) < (K,W)$ if $K\sqsubset H$, that is if $K$ is a strict subquotient of $H$. So the highest weight structure gives the fact that the projective indecomposable biset functors have a filtration by standard functors. This filtration has the following properties:
\begin{itemize}
\item  If $P_{H,V}$ denotes a projective cover of the simple $S_{H,V}$, then $P_{H,V}$ is filtered by a finite number of standard functors. The first quotient is $\Delta_{H,V}$, which appears with multiplicity one. The other standard objects which appear as subquotients are some $\Delta_{K,W}$ for $K\sqsubset H$.
\item Moreover the standard functors have finite length. The unique simple quotient of $\Delta_{H,V}$ is the simple functor $S_{H,V}$. The other simple functors which appear as composition factor of $\Delta_{H,V}$ are some $S_{K,W}$ for $H\sqsubset K$. 
\end{itemize}
\begin{de}
Let $k$ be a field. Let $G$ be a finite group. Let $\Sigma(G)$ be the full subcategory of $\mathcal{C}_{k}$ consisting of the subquotients of $G$. Then the group $G$ is call a $\mathrm{NV}_k$-group if the simple functors $S$ of $\mathcal{F}_{\Sigma(G),k}$ do not vanish at $G$.
\end{de}
It is well known that commutative groups are $\mathrm{NV}_{k}$-groups for every field $k$ (see Proposition $3.2$ of \cite{bst2}), but there are non-commutative $\mathrm{NV}_{k}$-groups. 
\begin{theo}\label{qh}
Let $G$ be a finite group. Let $k$ be a field such that $\mathrm{char}(k)$ does not divide $|\mathrm{Out}(H)|$ for all subquotients $H$ of $G$. If $G$ is a $\mathrm{NV}_{k}$-group, then $kB(G,G)$ is a quasi-hereditary algebra. 
\end{theo}
\begin{proof}
By Corollary $3.3$ of \cite{bst1}, the simple $kB(G,G)$-modules are exactly the evaluation at $G$ of the simple biset functors $S_{H,V}\in\mathcal{F}_{\Sigma(G),k}$. Now by Corollary $\ref{eval}$, if $P_{H,V}$ is a projective cover of $S_{H,V}$ in $\mathcal{F}_{\Sigma(G),k}$, then $P_{H,V}(G)$ is a projective cover of $S_{H,V}(G)$ as $kB(G,G)$-modules. Moreover since the standard functor $\Delta_{H,V}\in \mathcal{F}_{\Sigma(G),k}$ has a simple top, its evaluation at $G$ is indecomposable. 
\newline Let $M_{0}=0\subset M_{1}\subset \cdots \subset M_{n} = P_{H,V}$ be a standard filtration of $P_{H,V}$ in $\mathcal{F}_{\Sigma(G),k}$. The evaluation functor is exact, so the $kB(G,G)$-modules $M_{i}(G)$ produce a filtration of the projective indecomposable module $P_{H,V}(G)$. Moreover the quotient $M_{i}(G)/M_{i-1}(G)=[M_{i}/M_{i-1}](G)$ is the evaluation at $G$ of a standard functor indexed by a pair $(K,W)$ such that $K\sqsubset H$. It remains to look at the composition factors of the $\Delta_{H,V}(G)$. We have:
\begin{equation*}
\Delta_{H,V}(G)/\mathrm{Rad}(\Delta_{H,V}(G)) =[\Delta_{H,V}/\mathrm{Rad}(\Delta_{H,V})](G)=S_{H,V}(G). 
\end{equation*}
Moreover by Proposition $3.5$ of \cite{bst1}, a simple $kB(G,G)$-module $S_{K,W}(G)$ is a composition factor of $\Delta_{H,V}(G)$ if and only if $S_{K,W}$ is a composition factor of $\Delta_{H,V}$. As consequence $\Delta_{H,V}(G)$ has a simple top $S_{H,V}(G)$ and the other composition factors are some $S_{K,W}(G)$ for $H\sqsubset K$. This shows that $kB(G,G)\mb\Mod$ is a highest weight category in which the standard objects are the evaluation at $G$ of the standard functors of $\mathcal{F}_{\Sigma(G),k}$. 
\end{proof}
\begin{re}
This result can be easily generalized to the algebra $\mathrm{End}_{\mathcal{D}}(G)$ when $\mathcal{D}$ a admissible subcategory of $\mathcal{C}_{k}$. If the simple functors of $\mathcal{F}_{\mathcal{D}\cap \Sigma(G),k}$ do not vanish at $G$, then $\mathrm{End}_{\mathcal{D}}(G)$ is a quasi-hereditary algebra over a suitable field. 
\end{re}
As immediate corollary, for the double Burnside algebras, we have:
\begin{coro}
Let $k$ be a field such that $\mathrm{char}(k)$ does not divide $|\mathrm{Out}(H)|$ for all subquotients $H$ of a $\mathrm{NV}_{k}$-group $G$. Then the global dimension of $kB(G,G)$ is finite.  
\end{coro}
It should now be clear that the situation will not be that simple if some simple functors vanish at $G$. Indeed let $S_{H,V}$ be a simple functor of $\mathcal{F}_{\Sigma(G),k}$ such that $S_{H,V}(G)\neq 0$. If in a standard filtration of $P_{H,V}$ there is a standard functor $\Delta_{K,W}$ such that $S_{K,W}(G)=0$, then $\Delta_{K,W}(G)$ is \emph{not} in the set of standard modules for $kB(G,G)$ that we considered in the proof of Theorem \ref{qh}. In the rest of this paper we look at the case of $G=A_5$. We first show that the situation described here actually happens for this group, and we show that there is no hope to choose a better filtration for the projective $kB(G,G)$-modules. The reason is that $kB(G,G)$ has infinite global dimension. 
\section{The example of $A_5$.}
Let $k$ be a field of characteristic different from $2,3$ and $5$. The double Burnside algebra $kB(A_5,A_5)$ is a rather complicated object. Unfortunately it seems to the Author that $G=A_5$ is the smallest (or one of the smallest) example where the situation described above can appear. Indeed, this situation requires the existence of enough non-split extensions between simple functors in $\mathcal{F}_{\Sigma(G),k}$. It is well known that this category is not semi-simple if there are some non-cyclic groups in $\Sigma(G)$ (Theorem $1.1$ \cite{barker_ss}), but as it can be seen in Proposition 11.2 of $\cite{webb_strat2}$, if the category $\Sigma(G)$ does not contain enough increasing chains (for the subquotient relation) of objects, then there are not so many non-split extensions in $\mathcal{F}_{\Sigma(G),k}$. Moreover $A_5$ is also one of the first groups where the evaluation of the radical of the standard functor is not the radical of the evaluation (see Example $13.5$ of \cite{bst1}), so it is a good candidate for our purpose.  
\newline In order to simplify the computations, we will use the following results.
\begin{itemize}
\item Let $P_{K,W}$ be a projective indecomposable functor in $\mathcal{F}_{\Sigma(G),\field }$. Let $\Delta_{J,U}$ be a standard object in this category. Then
\begin{equation}\label{bgg}
[P_{K,W} : \Delta_{J,U}] = [\nabla_{J,U} : S_{K,W}]=[\Delta_{J,U^{*}} : S_{K,W^*}].
\end{equation}
Here $\nabla_{J,U}$ denotes the co-standard functor indexed by $(J,U)$. The first equality is the so-called BGG-reciprocity and the last equality follows from the usual duality in the biset-functor category. See Paragraph $8$ of \cite{webb_strat2} for more details. Note that for $A_5$ all the $\field \mathrm{Out}(H)$-simple modules that we will consider are self-dual. 
\item If $\mathcal{D}$ is an admissible full-subcategory of $\Sigma(G)$, then there is a restriction functor from $\mathcal{F}_{\Sigma(G),\field }$ to $\mathcal{F}_{\mathcal{D},\field }$. By Proposition $7.3$ of $\cite{webb_strat2}$, if $H\in \mathcal{D}$, then we have:
\begin{equation}\label{res}
[P_{H,V}^{\Sigma(G)} : \Delta_{K,W}^{\Sigma(G)} ]_{\Sigma(G)} = [P_{H,V}^{\mathcal{D}} : \Delta_{K,W}^{\mathcal{D}}]_{\mathcal{D}}. 
\end{equation}
\end{itemize}
\begin{lemma}\label{1}
Let $\field _{-}$ be the non-trivial simple $\field \mathrm{Out}(C_3)\cong \field \mathrm{Out}(A_4)\cong \field C_2$-module. There is a non-split exact sequence of functors of $\mathcal{F}_{\Sigma(A_5),\field }$:
\begin{equation*}
\xymatrix{
0\ar[r]& \Delta_{C_3,\field _{-}}\ar[r]& P_{A_4,\field _{-}}\ar[r] & \Delta_{A_4,\field _{-}}\ar[r]&0. 
}
\end{equation*}
\end{lemma}
\begin{proof}
We know that $P_{A_4,\field _{-}}$ has a finite $\Delta$-filtration with quotient $\Delta_{A_4,\field _{-}}$. We need to understand the other standard quotients of such a filtration. By the highest weight's structure of $\mathcal{F}_{\Sigma(A_5),k}$, such a standard quotient must be indexed by a subquotient of $A_4$. By using the BGG-reciprocity $(\ref{bgg})$ and formula $(\ref{res})$, a standard functor $\Delta_{H,V}$ appears in $P_{A_4,\field _{-}}$ if and only if $S_{A_4,\field _{-}}$ is a composition factor of $\Delta_{H,V}$ in $\mathcal{F}_{\Sigma(A_4),k}$. Using Proposition $3.5$ of \cite{bst1}, this is equivalent to the fact that $S_{A_4,k_{-}}(A_4)$ is a composition factor of $\Delta_{H,V}(A_4)$. As immediate consequence we have:
\begin{itemize}
\item $\Delta_{1,\field }$ is not in a $\Delta$-filtration of $P_{A_4,\field _{-}}$. Indeed $\Delta_{1,\field }$ is isomorphic to $\field B$, the usual Burnside functor. By the work of Bouc (see Section $5.4$ and $5.5$ of \cite{bouc_biset}), the simple subquotients of $\field B$ are the $S_{H,\field }$ for a $B$-group $H$. As consequence, the simple functor $S_{A_4,\field _{-}}$ is not a subquotient of $\field B$. 
\item $\Delta_{A_{4},\field }$ is not a subquotient of $P_{A_4,\field _{-}}$. Indeed, the only composition factor of $\Delta_{A_4,k}$ with $A_4$ as minimal group if $S_{A_4,k}$. 
\end{itemize}
We have the following: the subquotients of $A_4$ are : $A_4,V_4,C_3,C_2,1$.
\begin{enumerate}
\item $\mathrm{Out}(C_2)\cong 1$ and we have $\Delta_{C_2,k}(A_4)\cong S_{C_2,k}(A_4)$. 
\item $\mathrm{Out}(V_4)\cong S_3$. So there are three $k\mathrm{Out}(V_4)$-simple modules. We denote by $k$ the trivial module and $k_{-}$ the sign. Finally, we denote by $V$ the simple module of dimension $2$. Then we have: $\Delta_{V_4,k}(A_4)\cong S_{A_4,k}(A_4)$, $\Delta_{V_4,k_{-}}(A_4)\cong S_{A_4,k_{-}}(A_4)$ and $\Delta_{V_4,V}(A_4)=0$. 
\item $\mathrm{Out}(C_3)\cong C_2$ so there are two simple $k\mathrm{Out}(C_3)$-modules. We denote by $k$ the trivial module and $k_{-}$ the non trivial simple module. Then $\Delta_{C_3,k}(A_4)$ is a non-split extension between $S_{A_4,k}(A_4)$ and $S_{C_3,k}(A_4)$ and $\Delta_{C_3,k_{-}}(A_4)$ is a non-split extension between $S_{A_4,k_{-}}(A_4)$ and $S_{C_3,k_{-}}(A_4)$
\end{enumerate}
So the only standard functors which appear in a standard filtration of $P_{A_4,k_{-}}$ in $\mathcal{F}_{\Sigma(A_5),k}$ are $\Delta_{A_4,k_{-}}$ and $\Delta_{C_3,k_{-}}$. The structure of the highest weight category implies that $\Delta_{C_3,k_{-}}$ must be a subfunctor of $P_{A_4,k_{-}}$ and $\Delta_{A_4,k_{-}}$ must be a quotient of this functor. 
\end{proof}
Now we need to understand the evaluation at $A_5$ of $P_{A_4,k_{-}}$. 
\begin{lemma}\label{2}
\begin{itemize}
\item $\Delta_{A_4,\field _{-}}(A_5)\cong S_{A_4,\field _{-}}(A_5)\neq 0$.
\item $\Delta_{C_3,\field _{-}}(A_5)\cong S_{A_4,\field _{-}}(A_5)\neq 0$. 
\end{itemize}
\end{lemma}
\begin{proof}
The first isomorphism follows from the fact that $\Delta_{A_4,k_{-}}(A_5)$ is one dimensional, with basis $\mathrm{Ind}_{A_4}^{A_5}\otimes 1$. So it is a simple $kB(A_5,A_5)$-module of the form $S_{H,V}(G)$. The element $\Ind_{A_4}^{A_5}\Res^{A_5}_{A_4}$ acts by $1$ on $\Delta_{A_4,k_{-}}(A_5)$. So the minimal group $H$ is smaller than $A_4$. By the highest-weight structure of $\mathcal{F}_{\Sigma(A_5),k}$, the only possibility is to have $H=A_4$ and $V=k_{-}$. 
\newline We know that $\Delta_{C_3,k_-}$ is a subquotient of $P_{A_4,k_-}$, so $S_{A_4,k_-}$ is a composition factor of $\Delta_{C_3,k_-}$ by the BGG-reciprocity $(\ref{bgg})$. Since $S_{A_4,k_-}(A_5)\neq 0$, this simple module is a composition factor of $\Delta_{C_3,k_-}(A_5)$. Since we have $\dim_k \Big(\Delta_{C_3,k_-}(A_5)\Big)=1$, the result follows.  
\end{proof}
\begin{prop}
Let $G=A_5$ be the alternating group of degree $5$. Let $k$ be a field of characteristic different from $2,3$ and $5$. Then $\field B(G,G)$ has infinite global dimension. In particular $kB(G,G)$ is \emph{not} a quasi-hereditary algebra. 
\end{prop}
\begin{proof}
By using Lemma \ref{1}, we know that $P_{A_4,\field _{-}}$ has a $\Delta$-filtration with $\Delta_{A_4,\field _{-}}$ as quotient and $\Delta_{C_3,\field _{-}}$ as subfunctor. Since the simple quotient of $P_{A_4,\field_{-}}$ does not vanish at $G$, then $P_{A_4,\field_{-}}(A_5)$ is a projective cover of $S_{A_4,k_{-}}(A_5)$. By using Lemma \ref{2}, we see that this projective indecomposable functor is a non-split extension between $S_{A_4,\field_{-}}(A_5)$ and itself. 
\end{proof}
%\bibliographystyle{abbrv}
%\bibliography{mabib}

\par\noindent
{Baptiste Rognerud\\
EPFL / SB / MATHGEOM / CTG\\
Station 8\\
CH-1015 Lausanne\\
Switzerland\\
e-mail: baptiste.rognerud@epfl.ch}
\end{document}